\documentclass[12pt]{article}
\usepackage{amsmath,amssymb,color}
\usepackage{xspace}

\newcommand{\Var}{{\rm{Var}_{\mathbb{C}}}}

\newcommand{\Sym}{{\operatorname{Sym}}}
\newcommand{\diag}{{\mbox{\rm diag}}}

\def\uu{{\underline{u}}}

\def\tt{{\underline{t}}}

\def\kk{{\underline{k}}}

\def\1{\underline{1}}

\def\AA{{\mathbb A}}

\def\LLL{{\mathbb L}}
\def\Z{{\mathbb Z}}
\def\Q{{\mathbb Q}}
\def\C{{\mathbb C}}

\def\phiphi{{\underline{\varphi}}}
\def\rr{{\underline{r}}}

\newtheorem{theorem}{Theorem}

\newtheorem{lemma}{Lemma}
\newtheorem{proposition}{Proposition}

\newenvironment{definition}
{\smallskip\noindent{\bf Definition\/}:}{\smallskip\par}

\newenvironment{remark}
{\smallskip\noindent{\bf Remark\/}.}{\smallskip\par}

\newenvironment{proof}
{\noindent{\bf Proof\/}.}{{ $\square$}\smallskip\par}

\title{Higher order generalized Euler characteristics and generating series
\footnote{Math. Subject Class.: 32M99, 32Q55, 55M35. Keywords: complex quasi-projective varieties,
finite group actions, orbifold Euler characteristic, wreath products, generating series.}
}

\author{S.M.~Gusein-Zade \thanks{Partially supported by
the Russian government grant 11.G34.31.0005,
RFBR--13-01-00755,
NSh--4850.2012.1 and Simons-IUM fellowship.
Address: Moscow State University, Faculty
of Mathematics and Mechanics, GSP-1, Moscow, 119991, Russia. E-mail:
sabir\symbol{'100}mccme.ru} \and I.~Luengo \thanks{The last two authors are partially
supported by the grant MTM2010-21740-C02-01. Address: University
Complutense de Madrid, Dept. of Algebra, Madrid, 28040, Spain.
E-mail: iluengo\symbol{'100}mat.ucm.es} \and
A.~Melle--Hern\'andez \thanks{Address: 
ICMAT (CSIC-UAM-UC3M-UCM). Dept.\ of Algebra, 
Facultad de Ciencias Matem\'aticas, Universidad Complutense de Madrid, 
28040, Madrid, Spain.
E-mail: amelle\symbol{'100}mat.ucm.es}}
\date{}
\begin{document}
\def\eps{\varepsilon}

\maketitle

\begin{abstract}
 For a complex quasi-projective manifold with a finite group action, 
we define higher order generalized Euler characteristics 
with values in the Grothendieck ring of complex quasi-projective 
varieties extended by the rational powers of the class of the affine line.
We compute the generating series of generalized Euler characteristics 
of a fixed order of the Cartesian products of the manifold with the 
wreath product actions on them.
\end{abstract}

Let $X$ be a topological space (good enough, say, a quasi-projective variety) 
with an action of a finite group $G$.
For a subgroup $H$ of $G$, let $X^H=\{x\in X: Hx=x\}$ be the fixed point set of $H$. The orbifold Euler characteristic
$\chi^{orb}(X,G)$ of the $G$-space $X$ is defined, e.g., in \cite{AS}, \cite{HH}:
\begin{equation}\label{chi-orb}
 \chi^{orb}(X,G)=
\frac{1}{\vert G\vert}\sum_{{(g_0,g_1)\in G\times G:}\atop{\\g_0g_1=g_1g_0}}\chi(X^{\langle g_0,g_1\rangle})
=\sum_{[g]\in {G_*}} \chi(X^{\langle g\rangle}/C_G(g))\,,
\end{equation}
where $G_*$ is the set of conjugacy classes of elements of $G$, $C_G(g)=\{h\in G: h^{-1}gh=g\}$
is the centralizer of $g$, and $\langle g\rangle$ and $\langle g_0,g_1\rangle$ are the subgroups generated by the corresponding elements.

The higher order Euler characteristics of $(X,G)$ 
(alongside with some other generalizations) were defined
in \cite{BF}, \cite{T}.

\begin{definition}
 The {\em Euler characteristic} $\chi^{(k)}(X,G)$ {\em of order} $k$ of the $G$-space $X$ is
\begin{equation}\label{chi-k-orb}
 \chi^{(k)}(X,G)=
\frac{1}{\vert G\vert}\sum_{{{\bf g}\in G^{k+1}:}\atop{g_ig_j=g_jg_i}}\chi(X^{\langle {\bf g}\rangle})
=\sum_{[g]\in G_*} \chi^{(k-1)}(X^{\langle g\rangle}, C_G(g))\,,
\end{equation}
where ${\bf g}=(g_0,g_1, \ldots, g_k)$, $\langle{\bf g}\rangle$ is the subgroup 
generated by $g_0,g_1, \ldots, g_k$, and
$\chi^{(0)}(X,G)$ is defined as $\chi(X/G)$.
\end{definition}

The usual orbifold Euler characteristic $\chi^{orb}(X,G)$
 is the Euler characteristic of order $1$, $\chi^{(1)}(X,G)$.

 The higher order generalized Euler characteristics  takes values 
in the Grothendieck ring of complex quasi-projective 
varieties extended by the rational powers of the class of the affine line.
Let $K_0(\Var)$ be the Grothendieck ring of complex quasi-projective varieties. This is the abelian group
generated by the isomorphism classes $[X]$ of quasi-projective varieties modulo the relation:\newline
--- if $Y$ is a Zariski closed subvariety of $X$, then $[X]=[Y]+[X\setminus Y]$.\newline
The multiplication in $K_0(\Var)$ is defined by the Cartesian product. The class $[X]$ of a variety $X$
is the universal additive invariant of quasi-projective varieties and
can be regarded as a generalized Euler characteristic of $X$. Let $\LLL$ be the class $[\AA_{\C}^1]$
of the affine line and let $K_0(\Var)[\LLL^{1/m}]$ be the extension of the Grothendieck ring $K_0(\Var)$ by
all the rational powers of $\LLL$.

The formula for the generating series of the generalized orbifold Euler characteristics of the pairs $(X^n,G_n)$
in \cite{GLMSteklov} uses the (natural) power structure over the Grothendieck ring $K_0(\Var)$ (and over
$K_0(\Var)[\LLL^{1/m}]$) defined in \cite{GLM-MRL}. (See also \cite{GLM-Michigan}and 
\cite{GLMSteklov} for some generalizations of this
concept.) This means that for a power series 
$A(T)\in 1+t\cdot R[[t]]$ ($R=K_0(\Var)$ or $K_0(\Var)[\LLL^{1/m}]$)
and for an element $m\in R$ there is defined a series 
$\left(A(T)\right)^{m}\in 1+t\cdot R[[t]]$
so that all the properties of the exponential function hold. 
For a  quasi-projective variety $M$, the series $(1-t)^{-[M]}$
is the 
Kapranov zeta-function of $M$:
$\zeta_{[M]}(t):=(1-t)^{-[M]}=1+[M]\cdot t+[\Sym^2 M]\cdot t^2+ [\Sym^3M]\cdot t^3+\ldots,$
where $\Sym^k M=M^k/S_k$ is the $k$-th symmetric power of the variety $M$. 
A geometric description of the power structure over the 
over the Grothendieck
ring $K_0(\Var)$ is given in \cite{GLM-MRL} or \cite{GLMSteklov}.
The (natural) power structures over $K_0(\Var)$ and over $K_0(\Var)[\LLL^{1/m}]$ 
possess the following properties:
\begin{enumerate}
\item[1)] $\left(A(t^s)\right)^m =
\left(A(t)\right)^m\raisebox{-0.5ex}{$\vert$}{}_{t\mapsto t^s}$\;;
\item[2)] $\left(A(\LLL^s t)\right)^m=\left(A(t)\right)^{\LLL^sm}$\;.
\end{enumerate}

One can define a power structure over the ring $\Z[u_1,\ldots,u_r]$ of
polynomials in $r$ variables with integer coefficients in the following
way. Let $P(u_1,\ldots,u_r)=
\sum\limits_{\kk\in\Z_{\ge0}^r}p_\kk \uu^\kk\in \Z[u_1,\ldots,u_r]$,
where $\kk=(k_1,\ldots,k_r)$,
$\uu=( u_1,\ldots,u_r)$, $\uu^\kk =u_1^{k_1}\cdot\ldots\cdot u_r^{k_r}$,
$ p_\kk\in\Z$. Define
$$(1-t)^{-P(u_1,\ldots,u_r)}:=
\prod\limits_{\kk\in\Z_{\ge0}^r} (1-\uu^\kk t)^{-p_\kk},$$
where the power (with an integer exponent $-p_\kk$) means the usual one.
This gives a $\lambda$-structure on the ring $\Z[u_1,\ldots,u_r]$ and therefore
a power structure over it (see, e.g., \cite[Proposition 1]{GLMSteklov})

i.e., for polynomials $A_i(\uu)$, $i\ge1$, and $M(\uu)$, there is defined a
series $\left(1+A_1(\uu)t+A_2(\uu)t^2+\ldots\right)^{M(\uu)}$ with
the coefficients from $\Z[u_1,\ldots,u_r]$.

Let $r=2$, $u_1=u$, $u_2=v$. Let $e: K_0({\Var})\to {\Z}[u,v]$
be the ring homomorphism which sends the
class $[X]$ of a quasi-projective variety $X$ to its Hodge--Deligne
polynomial $e(X;u,v)=\sum h_X^{ij}(-u)^i(-v)^j$. 

\begin{remark}
Let $R_1$ and $R_2$ be rings with power structures over them. A ring
homomorphism $\varphi:R_1\to R_2$ induces the natural homomorphism
$R_1[[t]]\to R_2[[t]]$ (also denoted $\varphi$) by
$\varphi\left(\sum a_{i}\,t^{i}\right)=\sum\varphi(a_i)\,t^{i}$.
In \cite[Proposition~2]{GLMSteklov}, it was shown that 
if a ring homomorphism $\varphi:R_1\to R_2$ is such that
$(1-t)^{-\varphi(m)}=\varphi\left((1-t)^{-m}\right)$ for any $m\in R$,
then $\varphi\left(\left(A(t)\right)^m\right)=
\left(\varphi\left(A(t)\right)\right)^{\varphi(m)}$ for
$A(\tt)\in 1+t R[[t]]$, $m\in R$.
\end{remark}

\bigskip
There are two natural homomorphism from the Grothendieck ring $K_0(\Var)$
to the ring $\Z$ of integers and to the ring $\Z[u,v]$ of polynomials in
two variables: the Euler characteristic (with compact support)
$\chi:K_0(\Var)\to \Z$ and the Hodge--Deligne polynomial. Both 
possesses the following well known identities:
\newline (1) the formula of I.G.~Macdonald \cite{mac}:
$$
\chi(1+[X]t+[\Sym^2X]t^2+[\Sym^3X]t^3+\ldots)=(1-t)^{-\chi(X)},
$$
 (2) and the corresponding formula for the Hodge--Deligne polynomial 
(see \cite[Proposition 1.2]{Cheah1}):
$$
e(1+[X]t+[\Sym^2X]t^2+\ldots)=
(1-T)^{-e(X;u,v)}=\prod_{p,q}\left(\frac{1}{1-u^p v^q t}\right)^{e^{p,q}(X)}\,.
$$
These properties and the previous remark imply that the corresponding 
 homomorphisms respect the power structures over the corresponding
rings: $K_0(\Var)$ and $\Z[u,v]$ respectively, see 
\cite{GLM-Michigan}.

\medskip

A generalization of the orbifold Euler characteristic to the orbifold (or stringly) Hodge numbers and 
the orbifold Hodge--Deligne polynomial (for an action of a finite group $G$ on a non-singular quasi-projective
variety $X$) was defined in \cite{Vafa}, \cite{Zaslow}, \cite{BatDais}. 

 Let $X$ be a smooth quasi-projective variety
of dimension $d$ with an (algebraic) action of the group $G$. For $g\in G$, the centralizer $C_G(g)$ of $g$
acts on the manifold $X^{\langle g\rangle}$ of fixed points of the element $g$. Suppose that its action on the set of
connected components of $X^{\langle g\rangle}$ has $N_g$ orbits, and let
$X^{\langle g\rangle}_1$, $X^{\langle g\rangle}_2$, \ldots, $X^{\langle g\rangle}_{N_g}$
be the unions of the components of each of the orbits. At a point
$x\in X^{\langle g\rangle}_{\alpha_g}$, $1\le\alpha_g\le N_g$,
the differential $dg$ of the map $g$ is an automorphism of finite order of the tangent space $T_xX$. Its action
on $T_xX$ can be represented by a diagonal matrix $\diag(\exp(2\pi i \theta_1), \ldots, \exp(2\pi i \theta_d))$
with $0\le\theta_j<1$ for $j=1,2, \ldots, d$ ($\theta_j$ are rational numbers). The {\em shift number}
$F^{g}_{\alpha_g}$ associated with $X^{\langle g\rangle}_{\alpha_g}$ is
$F^{\langle g\rangle}_{\alpha_g}=\sum_{j=1}^{d}\theta_j\in\Q$. 
(It was introduced in \cite{Zaslow}.)

\begin{definition}\label{def-gen-orb}
 The {\em generalized orbifold Euler characteristic} of the pair $(X,G)$ (see \cite{GLMSteklov}) is
\begin{equation}\label{gen-orb}
 [X,G]=\sum_{[g]\in G_*}\sum_{\alpha_g=1}^{N_g}
[X^{\langle g\rangle}_{\alpha_g}/C_G(g)]\cdot\,\LLL^{F^{g}_{\alpha_g}}\in 
K_0(\Var)[\LLL^{1/m}]\,.
\end{equation}
\end{definition}

Since the Euler characteristic and the Hodge--Deligne polynomial  are additvie invariants 
they factor through $K_0(\Var)[\LLL^{1/m}]$ and 
the Euler characteristic morphisms sends $[X,G]$ to the orbifold Euler characteristic $\chi^{orb}(X,G)$.
The Hodge--Deligne polynomial morphism sends it to the orbifold Hodge--Deligne polynomial from
\cite{BatDais}, \cite{WZ}.

Let $G^n=G\times\ldots\times G$ be the Cartesian power of the
group $G$. The symmetric group $S_n$ acts on $G^n$ by permutation
of the factors: $s (g_1, \ldots, g_n) = (g_{s^{-1}(1)} , \ldots,
g_{s^{-1}(n)})$. The {\em wreath product} $G_n = G \wr S_n$ is
the semidirect product of the groups $G^n$ and $S_n$ defined by
the described action. Namely the multiplication in the group $G_n$
is given by the formula $({\bf g}, s)({\bf h}, t) = ({\bf g}\cdot s({\bf h}), st)$, 
where
${\bf g},\, {\bf h} \in G^n$, $s,\, t \in S_n$. The group $G^n$ is a normal
subgroup of the group $G_n$ via the identification of ${\bf g} \in G^n$
with $({\bf g},1) \in G_n$. For a variety $X$ with a $G$-action, there
is the corresponding action of the group $G_n$ on the Cartesian
power $X^n$ given by the formula
$$
( (g_1, \ldots, g_n), s)(x_1, \ldots, x_n) = (g_1 x_{s^{-1} (1)},
\ldots, g_n x_{s^{-1} (n)})\,,
$$
where $x_1, \ldots, x_n \in X$, $g_1, \ldots, g_n \in G$, $s\in
S_n$. One can see that the quotient $X^n/G_n$ is naturally
isomorphic to the space $\Sym^n (X/G)= (X/G)^n/S_n$. In particular, in the
Grothendieck ring of complex quasi-projective varieties one has 
$[X^n/G_n]=[(X/G)^n/S_n]=[\Sym^n(X/G)]$.

A formula for the generating series of the $k$-th order Euler characteristics of the pairs $(X^n, G_n)$
in terms of the $k$-th order Euler characteristics of the $G$-space $X$
was given in \cite{T} (see also \cite{BF}).

The generating series of the orbifold
Hodge--Deligne polynomials $e(X^n, G_n;u,v)$ of the pairs  $(X^n, G_n)$ was computed in \cite{WZ}.

A reformulation of the result of \cite{WZ} in terms of the generalized orbifold Euler characteristic
with values in $K_0(\Var)[\LLL^{1/m}]$ was given in \cite{GLMSteklov}.
Using properties of the power structure one has (\cite[Theorem~4]{GLMSteklov}):
\begin{equation}\label{eq5}
\sum_{n \geq 0} [X^n, G_n] t^n = \left(\prod_{ r =1}^{\infty} ( 1
-\LLL^{(r-1)d/2}t^r)\right)^{-[X, G]}\,.
\end{equation}

Here we define higher order generalized Euler characteristics of a pair $(X,G)$
(with $X$ non-singular) and give a formula for the generating series
of the $k$-th order generalized Euler characteristic of the pairs $(X^n,G_n)$.

\medskip

Before giving the definition of the higher order generalized Euler characteristic of a pair $(X,G)$
we discuss some versions of the definition (\ref{gen-orb}) and of the equation (\ref{eq5}).

For a $G$-variety $X$ (not necessarily
non-singular) its {\em inertia stack} (or rather {\em class}) $I(X,G)$  is defined by 
\begin{equation}\label{inertia}
I(X,G):=\sum_{[g]\in G_*} [X^g/C_G(g)]
\end{equation}
(see e.g. \cite{Kaw}, \cite{FLNU}). One can see that it is an analogue of the 
generalized orbifold Euler characteristic (\ref{gen-orb}) without the shift factor 
$\LLL^{F^{\langle g\rangle}_{\alpha_g}}$.
This inspires the following version of the definition (\ref{gen-orb}).

\begin{definition}\label{def-gen-orb-varphi1}
 For a rational number $\varphi_1$, let
\begin{equation}\label{gen-orb-wed}
[X,G]_{\varphi_1}:=\sum_{[g]\in G_*}\sum_{\alpha_g=1}^{N_g}
[X^{\langle g\rangle}_{\alpha_g}/C_G(g)]\cdot\,\LLL^{\varphi_1 F^{\langle g\rangle}_{\alpha_g}}\in 
K_0(\Var)[\LLL^{1/m}]\,.
\end{equation}
 \end{definition}
That is the Zaslow shift $F^{\langle g\rangle }_{\alpha_g}$ is multiplied by $\varphi_1$.
For $\varphi_1=1$ one gets the generalized Euler characteristic $[X,G]$ from (\ref{gen-orb}), for 
$\varphi_1=0$ one gets the inertia class $I(X,G)$.  The arguments from \cite{GLMSteklov}
easily give the following version of Equation (\ref{eq5}).

\begin{proposition}\label{gen-orb-mod}
 \begin{equation}\label{eq5-gen}
\sum_{n \geq 0} [X^n, G_n]_{\varphi_1} t^n = \left(\prod_{ r =1}^{\infty} ( 1
-\LLL^{\varphi_1 (r-1)d/2}t^r)\right)^{-[X, G]}\,.
\end{equation}
\end{proposition}

Thus multiplication  of Zaslow's  
shift by a number (at least by $1$ or $0$) makes sense.   
For the corresponding definition of the higher order generalized Euler characteristic
one can use factors $\varphi_k$ depending on the order of the Euler characteristic.

Let $X$ be a non-singular $d$-dimensional quasi-projective variety with a $G$ action and let  
$\underline{\varphi}=(\varphi_1,\varphi_2,\ldots)$ be a fixed sequence of rational numbers.
We use the notations introduced before (\ref{gen-orb}).
  
\begin{definition}\label{higher-gen-orb}
 The {\em generalized orbifold Euler characteristic of order} $k$ of the pair $(X,G)$ is
\begin{equation}\label{higher-gen}
 [X,G]^k_{\underline{\varphi}}:=\sum_{[g]\in G_*}
\sum_{\alpha_g=1}^{N_g}[X^{\langle g\rangle }_{\alpha_g}, C_G(g)]^{k-1}_{\underline{\varphi}}
\cdot\,\LLL^{\varphi_k F^{\langle g\rangle }_{\alpha_g}}\in 
K_0(\Var)[\LLL^{1/m}]\,,
\end{equation}
where  $[X,G]^1_{\underline{\varphi}}:=[X,G]_{\varphi_1}$ is the (modified) generalized 
orbifold Euler characteristic given by (\ref{gen-orb-wed}).
\end{definition}

\begin{remark}
The definition (\ref{chi-k-orb}) (as well as (\ref{chi-orb})) contains two equivalent versions.
One can say that here we formulate an analogues of the second one. A formula analogous to the first one
(with the factor $\frac{1}{\vert G\vert}$ in front) cannot work directly, at least without tensoring
the ring $K_0(\Var)[\LLL^{1/m}]$ by the field $\Q$ of rational numbers. Moreover, it seems that there
is no analogue of Theorem~\ref{main} in terms of the power structure. This gives the hint that
a definition of this sort makes small geometric sense (if any).
\end{remark}

Taking the Euler characteristic, one gets 
$\chi([X,G]^k_{\underline{\varphi}})=\chi^{(k)}(X,G)$.

To prove the formula for the generating series of 
$[X^n,G_n]_{\underline{\varphi}}^{k}$, we will use some
technical statements.

\begin{lemma}
 \begin{equation}
  [X'\times X'', G'\times G'']_{\phiphi}^{k}=[X', G']_{\phiphi}^{k}\times [X'', G'']_{\phiphi}^{k}\,.
 \end{equation}
\end{lemma}

The proof is obvious.

Let $X_1$ and $X_2$ be two $G$-manifolds and let $X_1^m\times X_2^{n-m}$ be embedded into $(X_1\coprod X_2)^n$ in the
natural way: a pair of elements $(x_{1,1}, \ldots, x_{1,m})\in X_1^m$ and $(x_{2,1}, \ldots, x_{2,n-m})\in X_2^{n-m}$
is identified with $(x_{1,1}, \ldots, x_{1,m}, x_{2,1}, \ldots, x_{2,n-m})\in (X_1\coprod X_2)^n$. Let
$\Sym^n(X_1^m\times X_2^{n-m})$ be the orbit of $X_1^m\times X_2^{n-m}$ under the $S_n$-action on $(X_1\coprod X_2)^n$.
The wreath product $G_n$ acts on $\Sym^n(X_1^m\times X_2^{n-m})$. 

\begin{lemma}
 \begin{equation}
[\Sym^n(X_1^m\times X_2^{n-m}), G_n]_{\phiphi}^{k}=
[X_1^m, G_m]_{\phiphi}^{k}\times [X_2^{n-m}, G_{n-m}]_{\phiphi}^{k}\,.
 \end{equation}
\end{lemma}

\begin{proof}
 An element $({\bf g}, s)\in G_n$ has fixed points on 
$\Sym^n(X_1^m\times X_2^{n-m})$ if and only if it is conjugate
to an element $({\bf g}', s')\in G_n$ such that $s'=(s_1, s_2)\in S_m\times S_{n-m}\subset S_n$ and the element
$({\bf g}', s')=(({\bf g}_1, {\bf g}_2), (s_1, s_2))$ has fixed points on $X_1^m\times X_2^{n-m}$ (and only on it).
The centralizer of the element $({\bf g}', s')$ is $C_{G_m}((g_1,s_1))\times C_{G_{n-m}}((g_2,s_2))$. 
The components
of $(X_1^m\times X_2^{n-m})^{\langle ({\bf g}', s')\rangle }$ are the products
$(X_1^m)^{\langle ({\bf g}_1, s_1)\rangle }_{\alpha}\times (X_2^{n-m})^{\langle ({\bf g}_2, s_2)\rangle }_{\beta}$
of the components of $(X_1^m)^{\langle ({\bf g}_1, s_1)\rangle }$ and $(X_2^{n-m})^{\langle ({\bf g}_2, s_2)\rangle }$. The shift
$F_{\alpha\beta}^{({\bf g}', s')}$ is equal to $F_{\alpha}^{({\bf g}_1, s_1)}+F_{\beta}^{({\bf g}_2, s_2)}$.
Therefore
\begin{eqnarray*}
&{}&
[\Sym^n(X_1^m\times X_2^{n-m}), G_n]_{\phiphi}^{k}{}\\
&{=}&\sum_{[({\bf g}', s')]}\sum_{\alpha\beta} [(X_1^m\times X_2^{n-m})_{\alpha\beta}^{({\bf g}', s')},
C_{G_m}(({\bf g}_1, s_1))\times C_{G_{n-m}}(({\bf g}_2, s_2))]_{\phiphi}^{k-1}
\cdot \LLL^{(F_{\alpha}^{({\bf g}_1, s_1)}+F_{\beta}^{({\bf g}_2, s_2)})}{}\\
&{=}&\sum_{[({\bf g}_1, s_1)]}\sum_{\alpha} [(X_1^m)_{\alpha}^{({\bf g}_1, s_1)},
C_{G_m}(({\bf g}_1, s_1))]_{\phiphi}^{k-1} \cdot \LLL^{F_{\alpha}^{({\bf g}_1, s_1)}}\times\\
&{}&
\sum_{[({\bf g}_2, s_2)]}\sum_{\beta} [(X_1^{n-m})_{\beta}^{({\bf g}_2, s_2)},
C_{G_{n-m}}(({\bf g}_2, s_2))]_{\phiphi}^{k-1} \cdot \LLL^{F_{\beta}^{({\bf g}_2, s_2)}}{}\\
&{=}&[X_1^m, G_m]_{\phiphi}^{k}\times [X_2^{n-m}, G_{n-m}]_{\phiphi}^{k}\,.
\end{eqnarray*}
\end{proof}

Let $X$ be a $G$-manifold and let $c$ be an element of $G$ acting trivially on $X$. Let $r$ be a fixed positive
integer. Denote by $G\cdot\langle a\rangle $ the group generated by $G$ and the additional element $a$ commuting with
all the elements of $G$ and such that $\langle a\rangle \cap G=\langle c\rangle $, $c=a^r$. Define the action of the group $G\cdot\langle a\rangle $
on $X$ (an extension of the $G$-action) so that $a$ acts trivially.

\begin{lemma}\label{lemma3} {\rm (cf. \cite[Lemma~4-1]{T})} In the described situation one has
$$
[X, G\cdot\langle a\rangle ]_{\phiphi}^{k}=r^k[X, G]_{\phiphi}^{k}\,.
$$
\end{lemma}

\begin{proof}
We shall use the induction on $k$. For $k=0$ this is obvious (since $[X, G]_{\phiphi}^{0}=[X/G]$).
Each conjugacy class of elements from $G\cdot\langle a\rangle $ is of the form $[g]a^s$, where $[g]\in G_*$,
$0\le s< r$. The fixed point set of $ga^s$ coincides with $X^g$, the Zaslow shift $F_{\alpha}^{ga^s}$
at each component of $X^g$ coincides with $F_{\alpha}^{g}$ (since $a$ acts trivially). The centralizer
$C_{G\cdot\langle a\rangle }(ga^s)$ is $C_G(g)\cdot\langle a\rangle $. Therefore
$$
[X, G\cdot\langle a\rangle ]_{\phiphi}^{k}
=\sum_{[g]\in G_*} r \sum_{\alpha=1}^{N_g} [X^g_{\alpha}, C_{G(g)}\cdot\langle a\rangle ]_{\phiphi}^{k-1} \cdot \LLL^{F_{\alpha}^{g}}
=r^k[X, G]_{\phiphi}^{k}\,.
$$
\end{proof}

\begin{theorem}\label{main}
Let $X$ be a smooth quasi-projective variety of dimension $d$ with a $G$-action. Then
\begin{equation}\label{Principal}
\sum_{n\ge 0}[X^n, G_n]_{\phiphi}^{k}\cdot t^n
=\left(\prod\limits_{r_1, \ldots,r_k\geq 1}\left(1-\LLL^{\Phi_k(\rr)d/2}\cdot t^{r_1r_2\cdots r_k}\right)^{r_2r_3^2\cdots r_k^{k-1}}\right)
^{-[X, G]_{\phiphi}^{k}}\,,
\end{equation}
where 
$$\Phi_k(r_1,\ldots,r_k)=\varphi_1(r_1-1)+\varphi_2r_1(r_2-1)+\ldots+
\varphi_kr_1r_2\cdots r_{k-1}(r_k-1).$$
\end{theorem}

\begin{proof}
To a big extend we shall follow the lines of the proof of Theorem~A in \cite{T}. We shall use
the induction on the order $k$. For $k=1$ the equation coincides with the one from Proposition~\ref{gen-orb-mod}.
Assume that the statement is proved for the generalized Euler characteristic of order $k-1$. One has
$$
\sum_{n\ge 0}[X^n, G_n]_{\phiphi}^{k}\cdot t^n
=\sum_{n\ge 0}t^n
\left(\sum_{[({\bf g},s)]\in G_{n*}}
\sum_{comp} [(X^n)^{\langle ({\bf g},s)\rangle }_{comp}, C_{G_n}(({\bf g},s))]_{\phiphi}^{k-1}
\cdot \LLL^{F_{comp}^{({\bf g},s)}}\right)\,,
$$
where the sums are over all the conjugacy classes $[({\bf g},s)]$ of elements of $G_n$ and over all the components
of $(X^n)^{\langle({\bf g},s)\rangle}$ (or rather unions of components from an orbit of the
$C_{G_n}(({\bf g},s))$-action on the components of it).

The conjugacy classes $[({\bf g},s)]$ of elements of $G_n$ are characterized by their types.
Let $a=({\bf g},s)\in G_n$, ${\bf g}=(g_1,\ldots, g_n)$. Let $z=(i_1,\ldots,
i_r)$ be one of the cycles in the permutation $s$. The {\em
cycle-product} of the element $a$ corresponding to the cycle $z$
is the product $g_{i_r}g_{i_{r-1}}\ldots g_{i_1}\in G$. The
conjugacy class of the cycle-product is well-defined by the
element ${\bf g}$ and the cycle $z$ of the permutation $s$. For $[c]\in G_*$
and $r\ge 0$, let $m_r(c)$ be the number of $r$-cycles in the
permutation $s$ whose cycle-products lie in $[c]$. One has
$$
 \sum\limits_{[c] \in G_*, r\geq 1} r m_r (c) = n\,.
$$
The collection  $\{m_r(c)\}_{r,c}$ is called the {\em type} of the element $a=({\bf g},s)\in G_n$.
Two elements of the group $G_n$ are conjugate to each other if and only if they are of the same type.

In \cite{T} (see also \cite{WZ}) it is shown that, for an element $({\bf g},s)\in G_n$ of type $\{m_r(c)\}$,
the subspace $(X^n)^{\langle ({\bf g},s)\rangle }$ can be identified with 
\begin{equation}\label{fixed}
\prod_{[c]\in G_*}\prod_{r\ge 1}(X^{\langle c\rangle })^{m_r(c)}\,.
\end{equation}
By \cite[Theorem~3.5]{T} the centralizer of the element $({\bf g},s)\in G_n$ is
isomorphic to
\begin{equation*}
\prod_{[c]\in G_*}\prod_{r\ge 1}\left\{(C_G(c)\cdot\langle a_{r,c}\rangle ) \wr S_{m_r(c)}\right\}
\end{equation*}
(acting on the product (\ref{fixed}) component-wise) where $C_G(c)\cdot\langle a_{r,c}\rangle $ is the group generated by
$C_G(c)$ and an element $a_{r,c}$ commuting with all the elements of $C_G(c)$ and such that
$a_{r,c}^r=c$, $\langle a_{r,c}\rangle \cap C_G(c)=\langle c\rangle $, and $a_{r,c}$ acts on $(X^{\langle c\rangle })^{m_r(c)}$ trivially.

The components of $( X^{\langle c\rangle })^{m_r(c)}$ (with respect to the $C_G(c)\cdot\langle a_{r,c}\rangle$-action) 
are $\Sym^{m_{r,c}}\left(\prod\limits_{\alpha=1}^{N_\alpha} 
(X_\alpha^{\langle c\rangle })^{m_{r,c}(\alpha)}\right),$
where $\sum\limits_{\alpha=1}^{N_\alpha} m_{r\nonumber\\
\allowbreak,c}(\alpha)=m_r(c).$ Here and bellow the sum over \emph{comp} means the summation 
over all the components indicated in the summands.  
Therefore
$$
\sum_{n\ge 0}[X^n, G_n]_{\phiphi}^{k}\cdot t^n
=\sum_{n\ge 0}t^n
\left(\sum_{[({\bf g},s)]\in G_{n*}} \sum_{comp} [(X^n)^{\langle({\bf g},s)\rangle}_{comp}, C_{G_n}(({\bf g},s))]_{\phiphi}^{k-1}
\cdot\LLL^{F_{comp}^{({\bf g},s)}}\right)
$$
\begin{eqnarray}
{}&=&\sum_{n\ge 0}t^n\cdot
\left(\sum_{\{m_r(c)\}}\sum_{comp}[\prod_{[c], r}\left\{ 
( X^{\langle c\rangle })^{m_r(c)}\right\}_{comp},
\prod_{[c],r}\left\{(C_G(c)\cdot\langle a_{r,c}\rangle ) \wr S_{m_r(c)}\right\}]_{\phiphi}^{k-1}
\cdot\LLL^{F_{comp}^{({\bf g},s)}}
\right)\nonumber\\
&{\ =}&\sum_{n\ge 0}t^n\cdot
\left(\sum_{\{m_{r,c}(\alpha)\}} \{ \prod_{[c], r} [\Sym^{m_{r,c}}
(\prod\limits_{\alpha=1}^{N_\alpha} 
(X_\alpha^{\langle c\rangle })^{m_{r,c}(\alpha)}),
(C_G(c)\cdot\langle a_{r,c}\rangle ) \wr S_{m_r(c)}]_{\phiphi}^{k-1} 
\times\right.\nonumber\\
&{\ }&\qquad\qquad\qquad\qquad\qquad\qquad\qquad\qquad\qquad\qquad\qquad\quad
\left.\LLL^{\phi_k (\sum\limits_{[c],r} \sum\limits_{\alpha=1}^{N_\alpha} 
m_{r,c}(\alpha) (F_{\alpha}^c+\frac{(r-1)d}{2}) 
) }\}\right)\nonumber
\end{eqnarray}
Iterating Lemma 2 one gets 
\begin{eqnarray}
&=&\sum_{\{m_{r,c}(\alpha)\}}t^{\sum\limits r m_{r,c}(\alpha)} \prod\limits_{[c], r} 
\left\{\prod\limits_{\alpha=1}^{N_\alpha} [(X_\alpha^{\langle c\rangle })^{m_{r,c}(\alpha)}, 
C_G(c)\cdot\langle a_{r,c}\rangle ) \wr S_{m_r(c)}]_{\phiphi}^{k-1}\times\right. \nonumber\\
&{\ }&\qquad\qquad\qquad\qquad\qquad\qquad\qquad\qquad
\left.\LLL^{\phi_k \left(\sum\limits_{[c],r} \sum\limits_{\alpha=1}^{N_\alpha} 
m_{r,c}(\alpha) (F_{\alpha}^c+\frac{(r-1)d}{2}) 
\right) }\right\}
\nonumber
\\
&{=}&\prod\limits_{[c], r} \left( \prod\limits_{\alpha=1}^{N_\alpha} \left(
\sum_{\{m_{r,c}(\alpha)\}}t^{r m_{r,c}(\alpha)} [(X_\alpha^{\langle c\rangle })^{m_{r,c}(\alpha)}, 
C_G(c)\cdot\langle a_{r,c}\rangle ) \wr S_{m_r(c)}]_{\phiphi}^{k-1}\times\right.\right. \nonumber\\
&{\ }&\qquad\qquad\qquad\qquad\qquad\qquad\qquad\qquad
\left.\left.\LLL^{\phi_k \left(\sum\limits_{[c],r} \sum\limits_{\alpha=1}^{N_\alpha} 
m_{r,c}(\alpha) (F_{\alpha}^c+\frac{(r-1)d}{2}) 
\right) }\right)\right)
\nonumber
\end{eqnarray}
By the induction one gets
\begin{eqnarray}
&=&\prod\limits_{[c], r} \prod\limits_{\alpha=1}^{N_\alpha} \left( 
\prod\limits_{r_1,\ldots,r_{k-1}\geq 1}
\left(1- \LLL^{\Phi_{k-1}(r)\frac{d}{2}}
(\LLL^{\varphi_k(F_{\alpha}^c+\frac{(r-1)d}{2})  } t^r)^{r_1\cdots r_{k-1}}
\right)^{r_2\cdot r_3^2\cdots  r_{k-1}^{k-2}  }
\right)^{-[X_\alpha^{\langle c\rangle },C_G(c)\cdot\langle a_{r,c}\rangle ]_{\phiphi}^{k-1}}
\nonumber
\\
& = &\hskip-10pt\left( \prod\limits_{r, r_1,\ldots,r_{k-1}\geq 1}
\left(1- \LLL^{\Phi_{k-1}(r)\frac{d}{2}}
\LLL^{\varphi_k(r_1\cdots r_{k-1}\frac{(r-1)d}{2})} t^{r_1\cdots r_{k-1}\cdot r}
\right)^{r_2\cdot r_3^2\cdots  r_{k-1}^{k-2}  }
\right)^{-\sum\limits_{[c],\alpha}[X_\alpha^{\langle c\rangle },C_G(c)\cdot\langle a_{r,c}\rangle ]_{\phiphi}^{k-1}
\LLL^{\phi_k F_{\alpha}^c}}
\nonumber
\end{eqnarray}
(Here we use the properties of the power structure.)
\begin{eqnarray}
&=&\left( \prod\limits_{r_1,\ldots,r_k\geq 1}
\left(1- \LLL^{(\Phi_{k-1}(r)+\varphi_k r_1\cdots r_{k-1}(r_k-1))\frac{d}{2}}  t^{r_1\cdots r_{k-1}\cdot r_k}
\right)^{r_2\cdot r_3^2\cdots  r_{k-1}^{k-2}  }
\right)^{-r_k^{k-1}\sum\limits_{[c],\alpha}[X_\alpha^{\langle c\rangle },
C_G(c)]_{\phiphi}^{k-1}
\LLL^{\phi_k F_{\alpha}^c}}\nonumber
\\
&{\ = }&\left(\prod\limits_{r_1, \ldots,r_k\geq 1}\left(1-\LLL^{\Phi_k(\rr)d/2}t^{r_1r_2\cdots r_k}\right)^{r_2r_3^2\cdots r_k^{k-1}}\right)
^{-[X, G]_{\phiphi}^{k}}\,.\nonumber
\end{eqnarray}
In the last two equations $r$ is substituted by $r_k$.
\end{proof}

\begin{remark}
For $\phiphi=\underline{0}$, i.e. if $\varphi_i=0$ for all $i$, the definition of the
higher order generalized Euler characteristics does not demand $X$ to be smooth. This way one gets
the definition of a sort of higher order inertia classes and the statement of Theorem~\ref{main} 
holds for an arbitrary $G$-variety $X$.
\end{remark}

Since $\chi ([X^n, G_n]_{\phiphi}^{k})=\chi^{(k)}(X,G)$, $\chi(\LLL)=1$, taking the Euler characteristic of
the both sides of the equation (\ref{Principal})
one gets Theorem~A of \cite{T}:
$$
\sum_{n\ge 0}\chi^{(k)}(X^n, G_n)\cdot t^n
=\left(\prod\limits_{r_1, \ldots,r_k\geq 1}\left(1-t^{r_1r_2\cdots r_k}\right)^{r_2r_3^2\cdots r_k^{k-1}}\right)
^{-\chi^{(k)}(X, G)}\,.
$$

Let $e^{(k)}_{\phiphi}(X, G; u,v):=e([X, G]_{\phiphi}^{k};u,v)$ be the
{higher order Hodge--Deligne polynomial} of $(X,G)$ (of order $k$). 
Applying the Hodge--Deligne polynomial homomorphism, one gets a generalization of the main result
in \cite{WZ}:
$$
\sum_{n\ge 0} e^{(k)}_{\phiphi}(X^n, G_n; u,v)\, \cdot t^n
=\left(\prod\limits_{r_1, \ldots,r_k\geq 1}\left(1-(uv)^{\Phi_k(\rr)d/2}\cdot t^{r_1r_2\cdots r_k}\right)^{r_2r_3^2\cdots r_k^{k-1}}\right)
^{-e^{(k)}_{\phiphi}(X, G;u,v)}\,.
$$

\end{document}